\documentclass{amsart}
\usepackage{amsfonts}
\usepackage{amsmath}
\usepackage{amsthm}
\usepackage{amssymb}
\usepackage{eufrak}
\usepackage{amscd}

\newcommand{\darkrad}{0.115}
\newcommand{\lrad}{0.25}
\newsavebox{\dEpic}
\savebox{\dEpic}(5,1){\begin{picture}(5, 1)
    \multiput(3,0.75)(1,0){2}{\circle*{\darkrad}}
    \multiput(3,0.25)(1,0){2}{\circle*{\darkrad}}
    \multiput(1,0.5)(1,0){2}{\circle*{\darkrad}}

    \put(1, 0.5){\line(1,0){1}}
    \put(3,0.75){\line(1,0){1}}
    \put(3, 0.25){\line(1,0){1}}

    \put(3,0.5){\oval(2,0.5)[l]}\end{picture}}
\newsavebox{\iEpic}
\savebox{\iEpic}(2.5,1){\begin{picture}(2.5,1)
    \multiput(0.25,0.25)(0.5,0){5}{\circle*{\darkrad}}
    \put(1.25,0.75){\circle*{\darkrad}}

    \put(0.25,0.25){\line(1,0){2}}
    \put(1.25,0.75){\line(0,-1){0.5}}

\end{picture}}
\newsavebox{\viipic}
\savebox{\viipic}(2.5,1){\begin{picture}(2.5,1)
    \multiput(0.25,0.25)(0.5,0){6}{\circle*{\darkrad}}
    \put(1.25,0.75){\circle*{\darkrad}}

    \put(0.25,0.25){\line(1,0){2.5}}
    \put(1.25,0.75){\line(0,-1){0.5}}

\end{picture}}
\newtheorem{thm}{Theorem}
\newtheorem{lemma}[thm]{Lemma}

\newtheorem{prop}[thm]{Proposition}

\newcommand{\qform}[1]{{\langle{#1}\rangle}}
\theoremstyle{definition}
\newtheorem{example}[thm]{Example}

\DeclareMathOperator{\tr}{tr} \DeclareMathOperator{\ad}{ad}
\DeclareMathOperator{\disc}{disc}
\DeclareMathOperator{\Br}{Br}
\DeclareMathOperator{\Inv}{Inv}
\DeclareMathOperator{\res}{res}
\DeclareMathOperator{\Hom}{Hom}
\DeclareMathOperator{\im}{im}
\DeclareMathOperator{\ch}{char}
\DeclareMathOperator{\spin}{Spin}
\begin{document}

\title{Killing Forms of Isotropic Lie Algebras}
\author{Audrey Malagon}
\address{Department of Mathematics, Mercer University, Macon, GA 31207}
\email{malagon\_al@mercer.edu}
\maketitle

\begin{abstract}
This paper presents a method for computing the Killing form of an isotropic Lie algebra defined over an arbitrary field based on the Killing form of a subalgebra containing its anisotropic kernel. This approach allows for streamlined formulas for many Lie algebras of types $E_6$ and $E_7$ and yields a unified formula for all Lie algebras of inner type $E_6$, including the anisotropic ones.
\end{abstract}

\section{Introduction}
The computation of Killing forms plays an important role in the classification of Lie algebras. The Killing form provides an invariant in the sense of Serre's lectures \cite{Coh} and has ties to other important invariants, so a method for computing them in a straightforward manner is a valuable tool. We are able to do this for Lie algebras over an algebraically closed field using the Cartan matrix. In this paper, we are primarily concerned with computing the Killing form of Lie algebras of exceptional type defined over an arbitrary field of characteristic not 2. N. Jacobson \cite{Jac2} and J-P. Serre \cite{Coh} give some results in this direction. Jacobson's results apply to certain Lie algebras of types $E_6$, $E_7$, $E_8$, $G_2$, and while they cover the exceptional Lie algebras over $\mathbb{R}$, his formulas are quite complicated, relying heavily on knowledge of the underlying Jordan and Cayley algebras. They also omit many cases over arbitrary fields. Serre gives the following straightforward formulas for $F_4$ and $G_2$ in terms of quadratic form invariants $q_3, q_5$ (See \cite{PR} or \cite{Coh} for precise definitions of these invariants. Here we use the notation of \cite{Lam}.)
\[\kappa_{G_2} = \qform{-1,-3}(q_3-1)\]
\[\kappa_{F_4} = \qform{-2}(q_5-q_3)+\qform{-1,-1,-1,-1}(q_3-1)\]

This paper presents similar streamlined formulas for other exceptional Lie algebras. Specifically, we describe a method which uses root system data to determine the Killing form of a isotropic Lie algebra in terms of the Killing form of a regular subalgebra. This method then provides a unified formula for all Lie algebras of inner type $E_6$ and produces clean, straightforward formulas for other Lie algebras of outer type $E_6$ and type $E_7$. While this paper focuses on Lie algebras whose anisotropic kernel contains a subalgebra of type $D_4$, the general method may be extended to other cases, particularly the Wesentlich case, by computing the Killing form on the anisotropic kernel consisting of algebras of type $A$. Other type two $E_6$ cases may also be studied with this method, specifically the case in which the anisotroipc kernel of type $^2A_5$ is a twisted form of $\frak{sl}_6$. Throughout the paper, the following numbering of vertices in the Dynkin diagrams will be used.
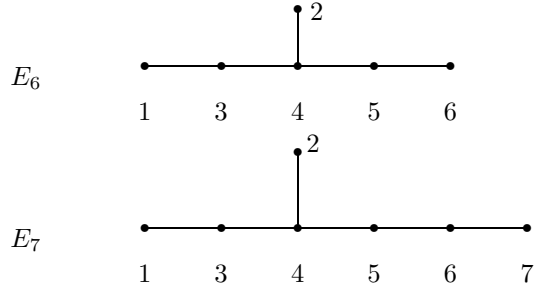
\newcommand{\rb}[1]{\raisebox{0.3in}[0pt]{#1}}
\begin{center}
\begin{tabular}{cc}
\rb{$E_6$} & \begin{picture}(7,2)
    \multiput(1,1)(1,0){5}{\circle*{\darkrad}}

    \put(3,1.75){\circle*{\darkrad}}

    \put(1,1){\line(1,0){4}}
    \put(3,1.75){\line(0,-1){.75}}

    \put(1,0.3){\makebox(0,0.4)[b]{$1$}}
    \put(2,0.3){\makebox(0,0.4)[b]{$3$}}
    \put(3,0.3){\makebox(0,0.4)[b]{$4$}}
    \put(4,0.3){\makebox(0,0.4)[b]{$5$}}
    \put(5,0.3){\makebox(0,0.4)[b]{$6$}}
    \put(3.25,1.6){\makebox(0,0.4)[b]{$2$}}
\end{picture}
\\
\rb{$E_7$} &\begin{picture}(7,2)
    \multiput(1,1)(1,0){6}{\circle*{\darkrad}}
    \put(3,2){\circle*{\darkrad}}

    \put(1,1){\line(1,0){5}}
    \put(3,2){\line(0,-1){1}}

    \put(1,0.3){\makebox(0,0.4)[b]{$1$}}
    \put(2,0.3){\makebox(0,0.4)[b]{$3$}}
    \put(3,0.3){\makebox(0,0.4)[b]{$4$}}
    \put(4,0.3){\makebox(0,0.4)[b]{$5$}}
    \put(5,0.3){\makebox(0,0.4)[b]{$6$}}
    \put(6,0.3){\makebox(0,0.4)[b]{$7$}}
    \put(3.2,2){\makebox(0,0.4)[b]{$2$}}
    \end{picture}
\end{tabular}
\end{center}
\begin{thm}\label{main}
Let $L$ be a simple, isotropic Lie algebra of dimension $n$ defined over a field $F$ with simple roots $\Delta$ (all of the same length) and Cartan subalgebra $H$. Let $\Delta_0$ denote the basis of simple roots for the anisotropic kernel of $L$, and let $A$ be a regular subalgebra of dimension $n'$ with simple roots $\Delta' \supset \Delta_0$. If $A = \displaystyle \oplus_{i=1}^{n} A_i$ with each $A_i$ simple, and $\ch F$ does not divide the Coxeter number of any $A_i$, then
the Killing form $\kappa$ on $L$ is given by

 \[\kappa = \qform{\frac{m(L)}{m(A_1)}}\kappa_1 \perp ... \perp \qform{\frac{m(L)}{m(A_n)}}\kappa_n  \perp \kappa|_{Z_H(A)} \perp
\frac{n-n'-|\Delta \setminus \Delta'|}{2}\mathcal{H}\]
where $m(-)$ is the Coxeter number of the algebra, $\kappa_i$ is the Killing form of $A_i$, and $Z_H(A)$ is the centralizer of $A$ in $H$.
\end{thm}
The results extend easily to Lie algebras whose roots have different lengths using a multiple that involves the dual Coxeter number rather than the Coxeter number of the algebra. The Killing form is computed on $Z_H(A)$ using a grading of the Lie algebra and results of \cite{ABS} to translate the question into one of dimensions of irreducible representations. Using this theorem, we also have the following formula for the Killing form of all inner type Lie algebras of type $E_6$, including the anisotropic ones, over any field of characteristic not 2 or 3. Here we refer to a modified Rost invariant $r(L)$ described in \cite{GG} and we use $e_3$ to denote the well known Arason invariant.
\begin{thm}\label{esix1}
The Killing form of a Lie Algebra of type $^1E_6$ is isomorphic to
\[ \qform{-1}4q_0 \perp \qform{2,6} \perp 24\mathcal{H} \]
where $1 \perp q_0$ is a 3-fold Pfister form and $e_3(1 \perp q_0)$ is $r(L)$.
\end{thm}
The final sections include explicit computations of Killing forms for certain Lie algebras of exceptional types.

\section{Isotropic Lie Algebras, Tits Indices \& Tits Algebras}\label{bases}
When studying Lie algebras over arbitrary fields, it is preferable to use Tits indices instead of Dynkin diagrams. A Tits index includes the Dynkin diagram of the root system together with information about split toral subalgebras and the action of the Galois group of $F$ on the Lie algebra. Tits indices were described and classified by J. Tits in \cite{TitsSS} for algebraic groups. (We will often make use of the theory of algebraic groups, which applies to these Lie algebras (see \cite[I.3]{LAG}, \cite[III.9]{LAG3}).)

We let $\Delta$ denote the collection of simple roots for a Lie algebra $L$ defined over a field $F$ with Cartan subalgebra $H$. We let $I$ denote the indexing set of $\Delta$. There is a natural action of the Galois group $\Gamma$ of $F_{sep}/F$ on $H^*$ since over $F_{sep}$, $H$ is split. The image of $\Delta$ under the action of an element $\sigma \in \Gamma$ will be another basis for $L$, and since the Weyl group permutes bases, we have a unique element $w$ in the Weyl group such that $w \circ \sigma(\Delta) = \Delta$. Tits defines the $*$-action of $\Gamma$ as the composition of the usual action with this element of the Weyl group. That is for $\sigma \in \Gamma$
\[\sigma^* := w \circ \sigma\]
The resulting action on $\Delta$ will be a graph automorphism of the Dynkin diagram. The Tits index is drawn from the Dynkin diagram by placing vertices of the Dynkin diagram which are in the same $*$-orbit close together in the Tits index. A Lie algebra is called \emph{inner} if the *-action is trivial and \emph{outer} if there are non-trivial orbits.

In addition to depicting the *-action of $\Gamma$ on $\Delta$, the Tits index also gives information about split toral subalgebras and the anisotropic kernel. The {(semisimple) anisotropic kernel} of a Lie algebra is roughly the part containing no split toral subalgebra. Precisely, it is the derived group of the centralizer of a maximal $F$-split toral subalgebra. In the Tits index, the simple roots which form the basis for the anisotropic kernel ($\Delta_0$) are uncircled. A Tits index in which all vertices are circled indicates that the Lie algebra is split, or in the outer case, quasi-split.

%
%
%
%
%
%
%
%

 \begin{example}\label{outer}
 The Tits index of a Lie algebra of type $^2\!E_6$ with anisotropic kernel of type $^2\!D_4$ is given below. The *-action permutes vertices $1,6$ and vertices $3,5$.

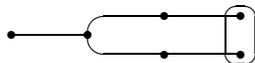
\begin{center}
{\begin{picture}(5, 1.1) \put(0,0){\usebox{\dEpic}}
\put(4,0.5){\oval(0.4,0.75)}
\end{picture}} \end{center}
\end{example}

It is important to note that the $*$-action of $\Gamma$ on the Tits index (and hence on the weight and root lattices) is not the same as the usual Galois action. In fact, elements fixed under the $*$-action may not be fixed under the usual Galois action and vice versa. We will often want to know the fixed elements under the usual Galois action since, for example, the elements of a (twisted) Cartan subalgebra fixed by $\Gamma$ form an $F$-split toral subalgebra. Let $\Lambda$ and $\Lambda_r$ denote the weight lattice and root lattice respectively and $\check{\Lambda}$, $\check{\Lambda_r}$ the weight and root lattices of the dual root system, also referred to as the co-weight and co-root lattices.

 Define the {fundamental dominant co-weights} as the dual basis to $\Delta$. These are the co-weights
\[\{\check{\lambda_j} \mid j \in I, (\check{\lambda_j}, \alpha_i) = \delta_{ij}\}\]
which form a basis for the co-weight lattice. Since $\check{\Lambda_r} \subset \check{\Lambda}$, each co-root $\check{\alpha}$ can be written as an integer combination of fundamental co-weights, and each fundamental co-weight can be written as a rational combination of simple co-roots. Let $c =|{\Lambda}/{\Lambda_r}|$. Then each $c\check{\lambda_j}$ can be written as an integer combination of co-roots. Suppose $c\check{\lambda_j} = \sum_{\check{\alpha} \in \check{\Delta}} a_i \check{\alpha_i} $. Then using the map ${\Phi} \leftrightarrow H$ sending ${\alpha} \rightarrow h_{\alpha}$, let
\[h_{c\check{\lambda_j}} = \sum a_i h_{\alpha_i}\]

 Borel and Tits \cite[Corollary 6.9]{BT} describe precisely which elements of $\check{\Lambda}$ are fixed under the Galois action and therefore which elements of $H$ are fixed under the Galois action. This allows us to construct a basis for $F$-split toral subalgebras from the Tits index.

Let $\{\alpha_{i_1},...,\alpha_{i_r}\}$ be a circled orbit in the Tits index (so that the $\alpha_{i_j}$ are not in the anisotropic kernel $\Delta_0$). Associate to this orbit the element $c\check{\lambda_{i_1}}+\cdots+c\check{\lambda_{i_r}}  \in \check{\Lambda}_r$ where $c = |{\Lambda}/{\Lambda_r}|$. Then the subspace of the co-root lattice fixed under the usual Galois action is generated over $\mathbb{Q}$ by elements of the form
\[\{c\check{\lambda_{i_1}}+\cdots+c\check{\lambda_{i_r}} \mid \text{ $\{\alpha_{i_1},...,\alpha_{i_r}\}$ is a Tits orbit in $\Delta \setminus \Delta_0$}\}\]
 \cite[Corollary 6.9]{BT}.

 In the case that the $*$-action on the Tits index is trivial, all orbits contain only one vertex and \[\check{\Lambda}^{\Gamma} \otimes \mathbb{Q} = \qform{c\check{\lambda_i} \mid \alpha_i \in \Delta \setminus \Delta_0}. \]
An $F$-split toral subalgebra has basis given by elements of the form
\begin{equation}\label{basis1}
\{h_{c\check{\lambda_i}} \mid \text{vertex $i$ is circled in the Tits index}.\}
\end{equation}
In particular, the dimension of a maximal $F$-split toral subalgebra is equal to the number of circled vertices in the Tits index.

Suppose now that $L$ is type 2. The $*$ action gives a map
\[*: \Gamma \rightarrow \text{Aut} \Delta.\]
The field $K = (F_{sep})^{\ker(*)}$ is a finite extension of $F$ of degree $|\text{im}(*)| = 2$ for which the Tits Index of $L \otimes K$ is of inner type. Letting $K = F(\sqrt{a})$ and using \cite[p.279]{LAG2}, a basis for the maximal $F$-split toral subalgebra in $L$ is given by
\begin{equation}\label{basis2}
\{h_{c\check{\lambda_i}} + h_{\check{c\lambda_j}}, \sqrt{a}h_{c\check{\lambda_i}}-\sqrt{a}h_{c\check{\lambda_j}} \mid  \text{$\{\alpha_i, \alpha_j\}$ \footnotesize{a circled Tits orbit}} \}
\end{equation}

In the case that all roots have the same length, weights and co-weights are equal. Formulas for the fundamental dominant weights written as sums of the simple roots can be found in \cite[Sec. 13, Table 1]{Hum}.


We also have associated to each orbit in the Tits index a central simple algebra known as a Tits Algebra. These were defined by Tits in \cite{Tits1971} and are also described in \cite{Anne}. These central simple algebra invariants give us information about the Lie algebra which will be necessary for the results on Killing forms. Tits defines these algebras by giving a bijection between dominant weights fixed under the $*$-action and algebra representations of $L$. An algebra representation of a Lie algebra $L$  defined over $F$ is simply a map $\rho: L \rightarrow GL_{1}(A)$ for a central simple algebra $A$ over $F$ (see \cite{Anne}).

It is well known that there exists a bijection between dominant weights of a split Lie algebra and irreducible representations of the Lie algebra (\cite[Lemma 2.2]{Tits1971}). Let
\[\beta: \Lambda_+ \rightarrow \text{irreducible representations of $L$}.\]
Tits extends this notion to non-split algebras by restricting $\beta$ to just the elements of $\Lambda_+$ fixed under the $*$-action of $\Gamma$. In this case by \cite[Theorem 3.3]{Tits1971}, we have a bijection
\[\Lambda_+^{\Gamma} \leftrightarrow \{\text{irreducible algebra representations of $L$}\}\]
If $\lambda$ is in the root lattice of $L$, then the the Tits algebra associated to $\lambda$ is split.
Furthermore, the Brauer class of the Tits algebra associated to the sum of weights $\lambda+\mu$ is the same as the sum in the Brauer group of the Tits algebras associated to the weights $\lambda$ and $\mu$. These properties allow us to define a homomorphsim
\[\alpha: (\Lambda / \Lambda_r)^{\Gamma} \rightarrow \Br(F)\]
and define the Tits algebra of a weight only up to Brauer equivalence. For any $\lambda \in \Lambda_+$ associated to a $*$-orbit in the Tits index we write $A(\lambda)$ for its Tits algebra. In the case of an inner type Lie algebra, we may associated a Tits algebra to each vertex and we deonte the Tits algebra associated to vertex $i$ by $A(i)$.

\section{Weight Space Decomposition}

For this section $S$ is a (not necessarily maximal) $F$-split toral subalgebra, which sits inside a Cartan subalgebra $H$. It acts on the Lie algebra via the adjoint representation, and we can decompose the Lie algebra into the eigenspaces induced by this action just as we would for the action of $H$ in the split case. We use $L_{\mu}$ to denote the weight space associated to the weight $\mu$:
\[L_{\mu} := \{ x \in L  \mid \text{$ [hx] = \mu(h)x$ for all $h \in S$} \}\]

Just as we had a highest root, we will also have a {highest weight}.
Using the weight spaces, $L$ decomposes as

\[L = L_{0} \oplus \bigoplus_{\mu \in S^*, \mu \neq 0} (L_{\mu} \oplus L_{-\mu}) \]
A weight space decomposition is similar to a root space decomposition in that the nonzero weight spaces occur in positive and negative pairs. If $\mu$ is a weight, $-\mu$ is also a weight, and furthermore $L_{\mu}$ and $L_{-\mu}$ have the same dimension. It is not true, however, that $L_{\mu}$ must be one dimensional. In addition, the zero weight space is not the same as the zero root space. It will contain the Cartan subalgebra $H$, but it is often larger than $H$ alone. The weight space decomposition is useful for computing the Killing form because the properties of weights mirror the advantageous properties of roots.

\begin{lemma} \label{rootform} Let $L$ be a simple Lie algebra over a field $F$ with Chevalley basis $\{h_{\alpha_i} \mid \alpha_i \in \Delta\} \cup \{ x_{\alpha} \mid \alpha \in \Phi^+\} \cup \{x_{-\alpha} \mid \alpha \in \Phi^+\}$ over an algebraic closure of $F$. If $\mu$, $\lambda$ are weights of $L$ such that $\mu+\lambda \neq 0$, then $L_{\mu}$ is perpendicular to $L_{\lambda}$ relative to the Killing form $\kappa$ of $L$. Furthermore, $L_{\mu} \oplus L_{-\mu}$ is hyperbolic with respect to $\kappa$. If $L$ is split, the Killing form $\kappa$ is given by the Weyl-invariant inner product on the dual root space. Specifically
\[\kappa(h_{\alpha}, h_{\beta}) = 2m^*(L)(\check{\alpha}, \check{\beta})\]
where $(\check{\alpha}, \check{\beta})$ is the Weyl-invariant inner product with $(\check{\alpha}, \check{\alpha}) = 2$ for a long root  $\alpha$ and $m^*(L)$ is the dual Coxeter number of the algebra.
\end{lemma}
\begin{proof}
Let $\mu, \lambda$ be weights of $L$. Since $\mu + \lambda \neq 0$, we must have $h \in H$ such that $(\mu + \lambda)(h) \neq 0$.  Let $v_1 \in L_{\mu}$ and $v_2 \in L_{\lambda}$. The Killing form is associative so we have
\begin{align*}
\mu(h)\kappa(v_1, v_2) &= \kappa([hv_1],v_2)=   -\kappa([v_1h],v_2)\\ &=  -\kappa(v_1,[hv_2])= -\lambda(h)\kappa(v_1,v_2) \\
(\mu(h)+\lambda(h))\kappa(v_1,v_2) &= 0
\end{align*}
Since $(\mu+\lambda)(h) = \mu(h) + \lambda(h) \neq 0$, we must have
\[\kappa(v_1,v_2) = 0 \]
Furthermore, since $\mu + \mu \neq 0$, $L_{\mu}$ is totally
isotropic with respect to $\kappa$.  Since $\dim(L_{\mu}) =
\frac{1}{2}\dim(L_{\mu} \oplus L_{-\mu})$ and $L_{\mu} \oplus
L_{-\mu}$ is non-degenerate, $\kappa$ restricts to be hyperbolic on $L_{\mu} \oplus L_{-\mu}$ (\cite[I.3.4(1)]{Lam}). In the split case, our weights are just the roots. Let $\check{\alpha}, \check{\beta} \in \check{\Phi}$. Define a symmetric bilinear form $f$ on $\check{\Phi}$ by
\[f(\check{\alpha}, \check{\beta}) = \kappa (h_{\alpha}, h_{\beta})\]
This form is Weyl-invariant. (By \cite[Lemma 9.2]{Hum}, $\qform{\alpha,\beta}$ is Weyl-invariant, and since the Weyl group action preserves root length, this implies the inner product on the root space $(\alpha, \beta)$ is Weyl-invariant. But $\sigma_{\check{\alpha}}(\check{\beta}) = \check{(\sigma_{\alpha}{\beta})}$, which implies that the form $f$ on the dual root system is also Weyl-invariant.)
If $\Phi$ is an irreducible root system, then $\Phi$ is an irreducible representation of the Weyl group. If not, $\Phi$ decomposes uniquely as a direct sum of irreducible root systems (\cite[Prop 6, VI.1.2]{Bbk}) so it suffices to work in the irreducible case. Then Schur's lemma states that there is at most one Weyl-invariant symmetric form on $\Phi$ up to scalars. Take the scalar multiple of $f$ that gives $(\check{\alpha}, \check{\alpha}) = 2$ for a long root $\alpha$ so that the form may be computed from the literature (see for example \cite{Bbk}). By \cite[p.14]{Cox} for a long root $\alpha$
\[\kappa(h_{\alpha},h_{\alpha}) = 4 m^*(L)\]
Therefore
\[\kappa(h_{\alpha},h_{\alpha}) = 2 m^*(L)(\check{\alpha}, \check{\alpha}) \qedhere \]

\end{proof}

\textbf{Note:} In the case of a simply laced root system, $\check{\alpha} = \alpha$ and $\qform{\alpha,\beta} = (\alpha,\beta)$ so over an algebraically closed field the Killing form on $H$ is given by the Cartan matrix.

The next lemma describes the zero weight space.

\begin{lemma} \label{zerows}
Let $S$ be an $F$-split toral subalgebra of $L$ and let $A$ be the derived subalgebra of $Z_{L}(S)$. Let $H$ be a (not necessarily split) Cartan subalgebra of $L$. Then in the weight space decomposition of $L$ with respect to $S$
\[L_0 = Z_H(A) \oplus A\]
where $\oplus$ is an orthogonal sum with respect to the Killing form. Furthermore, $A$ is semisimple and $A$ contains the anisotropic kernel of $L$.
\end{lemma}

\begin{proof}
It is clear that $L_0 = Z_L(S)$. Since $L$ is semisimple, $L_0$ is the centralizer of a toral subalgebra, so $L_0$ is reductive \cite[Corollary A 26.2]{LAG3} and by \cite[Theorem 11]{Jac} decomposes as
\[L_0 = Z(L_0) \oplus [L_0L_0]\]
By definition we have $A = [L_0L_0]$. Let us now examine $Z(L_0)$. Let $x \in L_0$. Clearly $[xZ(L_0)]=0$ so $x \in Z(L_0)$ if and only if $[xA] = 0$. By  \cite[Theorem 11]{Jac} $Z(L_0)$ is a toral subalgebra and is therefore contained in a maximal toral subalgebra. Since all maximal toral subalgebras are conjugate, $Z(L_0)$ is contained in all maximal toral subalgebras, and in particular $Z(L_0)$ is contained in $H$. Therefore $Z(L_0)$ is precisely the elements of $H$ which centralize $A$.
\[Z(L_0) = Z_H(A)\]
If $S$ is a maximal $F$-split toral subalgebra, $A$ is precisely the semisimple anisotropic kernel of $L$ with simple roots corresponding to the non-circled vertices in the Tits index  (\cite{TitsSS}). If $S$ is contained in a maximal $F$-split toral subalgebra, then $A$ will contain the semisimple anisotropic kernel.
\end{proof}

Lemmas \ref{rootform} and \ref{zerows} give the following decomposition for any isotropic Lie algebra $L$ with split toral subalgebra $S$ and subalgebra $A$ containing the anisotropic kernel.
\[L = A \oplus Z_H(A) \oplus \bigoplus_{\mu \in S^*, \mu \neq 0} (L_{\mu} \oplus L_{-\mu})\]
where the sums oustide the parenthesis are orthogonal with respect to $\kappa$

\section{Method for Computing Killing Forms}

The weight space decomposition of $L$ given in the previous section significantly streamlines the computation of the Killing form. With this decomposition in mind, we now prove Theorem \ref{main}
\begin{proof}

Since $A$ contains the anisotropic kernel of $L$, we know that $Z_H(A)$ is $F$-split. Decomposing with respect to $Z_H(A)$ gives
\[L = A \oplus Z_H(A) \oplus \bigoplus_{\mu \in S^*, \mu \neq 0} (L_{\mu} \oplus
L_{-\mu})\] where the sums outside the parenthesis are orthogonal
with respect to $\kappa$.  Furthermore, each $A_i$ is orthogonal with respect to the Killing form so $\kappa|_A$ is the sum of the $\kappa|_{A_i}$. Let $\kappa_i$ denote the Killing form of the subalgebra $A_i$. We first note that $\kappa|_{A_i}$ is $A_i$-invariant by the properties
of trace. Since $\kappa_i$ is also $A_i$-invariant, and the nonzero
$A_i$-invariant bilinear form on $A_i$ is unique up to a scalar multiple
(\cite[Theorem 5.1.21]{GW}), we must have
\[\kappa|_{A_i} = \qform{c} \kappa_i\]
for some scalar $c$. Now let $\alpha \in \Delta' \subset \Delta$ so
$h_{\alpha} \in A_i$.
Then by Lemma \ref{rootform}
\[\kappa(h_{\alpha},h_{\alpha}) = 4 m(L) \quad\text{and}\quad \kappa_i(h_{\alpha},h_{\alpha}) = 4m(A_i)\]
forcing $c =\frac{m(L)}{m(A_i)}$. That is,
 \[\kappa|_A = \frac{m(L)}{m(A_1)}\kappa_1 \perp ... \perp \frac{m(L)}{m(A_l)}\kappa_l.\]

On the subspaces  $L_{\mu_i} \oplus
L_{-\mu_i}$, we know that $\kappa$ restricts to be hyperbolic (Lemma \ref{rootform}) and the dimension of the nonzero (hyperbolic) weight space is $\dim{L} - \dim{L_0} = n - (n'+|\Delta \setminus \Delta'|)$. This gives the result.

\end{proof}
\textbf{Note:} In the case that $L$ has roots of different lengths, the above theorem holds by simply replacing the Coxeter number with the dual Coxeter number, except in the following case. Suppose there exists $\alpha \in \Delta' \subset \Delta$ such that $\alpha$ is short in $A_i$ but long in $L$. Then according to \cite[p.14-15]{Cox}
\[\kappa'(h_{\alpha},h_{\alpha}) = 4m^*(A_i)\]
\[\kappa (h_{\alpha}, h_{\alpha}) = 4cm^*(L)\]
where $c$ is the square of the ratio of the root lengths and $m^*$ is the dual Coxeter number of the algebra.
Then the scalar is $\frac{cm^*(L)}{m^*(A_i)}$.

The only remaining subspace of the decomposition for which we have not computed the Killing form is $Z_H(A)$. The next section discusses computing the Killing form on a toral subalgebra by calculating dimensions of appropriate irreducible representations.

\section{Calculating the Killing form on $Z_H(A)$}

In this section we discuss a streamlined method for computing the Killing form on a split toral subalgebra, specifically the centralizer of a subalgebra containing the anisotropic kernel. We will assume our root system is simply laced so that coroots equal roots and coweights equal weights. Since $A$ contains the anisotropic kernel of $L$, $Z_H(A)$ is split or quasi-trivial. We use the results of Section \ref{bases} to see the basis of a maximal $F$-split toral subalgebra containing $Z_H(A)$. Since the Killing form on a toral subalgebra of outer type can be computed in terms of basis elements for the inner type toral subalgebra, we will work only in the case that $L$ is of inner type and $\Phi$ is simply connected.

\begin{prop}
Let  $A$ have basis $\Delta' \subset \Delta$. Then a basis for $Z_H(A)$ is

\[\{ h_{c\lambda_j} \mid \alpha_j \in \Delta \setminus \Delta'\} \]

\end{prop}

\begin{proof} A basis for $Z_H(A)$ is contained in the basis \ref{basis1}. Let $h_{c\lambda_j}$ be a basis element from \ref{basis1}. Then $h_{c\lambda_j} \in Z_H(A)$ if and only if
\[[h_{c\lambda_j}x_{\alpha_k}] = \qform{\alpha_k, c\lambda_j} = 0\]
for all $\alpha_k \in \Delta'$. For each such $\alpha_k$ define
\[ \phi_k: {\Lambda}_r \rightarrow \mathbb{Z}: \alpha_k \rightarrow \qform{\alpha_k, \lambda}\]
Then $h_{\lambda} \in Z_H(A)$ if and only if ${\lambda} \in \displaystyle{\bigcap_k} \ker \phi_k$. Clearly each basis element given above is in $\bigcap_k \ker \phi_k$, and from Section \ref{bases} we know the rank of $Z_H(A)$ is $|\Delta \setminus \Delta'|$.
\end{proof}
We can compute the Killing form on $Z_H(A)$ using this basis. Let $h_{c{\lambda_j}} = \sum_{\alpha_i \in \Delta}a_ih_{\alpha_i}$ First note that
\[[h_{c{\lambda_j}} x_{\alpha}] = \alpha(h_{c{\lambda_j}})x_{\alpha} = \qform{\alpha, c{\lambda_j}}
 c a_jx_{\alpha}\]
so $\ad(h_{c\lambda_j})$ is a diagonal matrix with
entries $c$(coefficient of $\alpha_j$). Using the symmetry of positive and negative roots we have
\begin{equation}\label{kappaij}
\tr(\ad(h_{c\lambda_i})\ad(h_{c\lambda_j})) = \\
 2c^2 \sum_{\alpha \in \Phi+}
(\text{coefficient of $\alpha_i$ in $\alpha$})(\text{coefficient of
$\alpha_j$ in $\alpha$})
\end{equation}

or in the case $i=j$
\begin{equation}\label{kappaii}
\tr(\ad(h_{c\lambda_i})\ad(h_{c\lambda_i})) = 2c^2\sum_{\alpha \in \Phi+}
(\text{coefficient of $\alpha_i$ in $\alpha$})^2
\end{equation}
To compute $\kappa(h_{c\lambda_i}, h_{c\lambda_j})$, we need only those roots in $\Phi^+$ with nonzero $\alpha_i$ and $\alpha_j$ coefficients. To count these roots we introduce the notation of \cite{ABS}.

Fix a subset $\Delta_J \subset \Delta$. The \emph{level} of a positive root $\beta = \displaystyle \sum_{i \in I} c_i\alpha_i$ with respect to $\Delta_J$ is the sum of the coefficients of the $\alpha_i \in \Delta \setminus \Delta_J$. The \emph{shape} of $\beta$ is $\sum_{i \in I\setminus J} c_i \alpha_i$. The goal is to grade the Lie algebra according to the level and shape of its roots. Then the Killing form of an element $h_{n\lambda_i}$ can be given in terms of the dimension of an irreducible representation of a subalgebra.

Let $M_{I\setminus J}(l)$ denote the product of all root spaces $L_{-\beta}$ with $\text{level}(\beta) = l$. By \cite[Theorem 2]{ABS},
\[M_{I\setminus J}(l) = \prod_{\text{$S$ a shape of level $l$}} V_S \]
where each $V_S$ is an irreducible representation of the Lie algebra $L_J$ generated by basis $\Delta_J$.

Furthermore, the highest weight of $V_S$ is the negative of the root with shape $S$ and minimal height (in $L$) (\cite{ABS}). In the case $\Delta \setminus \Delta_J = \alpha_i$, the only possible shape for a fixed level $l$ is $l\alpha_i$ so $M_{i}(l)$ is a standard cyclic representation of $L_J$ with highest weight $-l\alpha_i - \beta_l$, denoted $V_{L_J}(-l\alpha_i - \beta_l)$. The dimension of this representation can be used to compute $\kappa(h_{n\lambda_i}, h_{n\lambda_i})$.
\begin{prop}\label{dim}
\[\kappa(h_{n\lambda_i},h_{n\lambda_i}) = 2n^2 \sum_{l>0} l^2
\dim V_{L_J}(-l\alpha_i-\beta_l)\]
\end{prop}
\begin{proof}
\begin{align*}
\tr(\ad(h_{n\lambda_i})\ad(h_{n\lambda_i})) &= 2n^2\sum_{\beta \in
\Phi^+} (\text{coefficient of $\alpha_i$ in $\beta$})^2\\
 &= 2n^2
\sum_{l>0} l^2 \dim_{L_J} M_{\{i\}}(l). \qedhere
\end{align*}
\end{proof}
The same method can be used in some cases to compute $\kappa(h_{n\lambda_i},h_{n\lambda_j})$ for $i \neq j$. In this case, take $\Delta_J = \Delta \setminus \{\alpha_i, \alpha_j\}$.
A root $\beta$ contributes to the above sum if only if both its $\alpha_i$ and $\alpha_j$ coefficients are nonzero. In the $E_6$ case with $L' = D_4$, this condition states $level(\beta) = 2$ with respect to $\Delta \setminus \Delta_J = \{\alpha_1, \alpha_6\}. $
The notation makes this process appear more complicated than necessary. Using it to compute the Killing forms for Real Lie algebras in the next section should clarify that this is really a straightforward process.
\section{Real Lie Algebras}
Theorem $\ref{main}$ produces very nice results for Lie algebras over the real numbers. It allows us to read the Killing form directly from the Tits index by noting the number of single circled vertices and the size of the anisotropic kernel.

\begin{lemma}\label{orbit}
Let $L$ be a Lie algebra of type $2$ over $F$ and let $\{\alpha_i, \alpha_j\}$ be an orbit in the Tits index of $L$ not contained in $\Delta_0$. The Killing form on the 2 dimensional toral subalgebra associated to this orbit is given by
\[\qform{2(x+y), 2a(x-y)}\]
where $x$ and $y$ are integers such that $x>y>0$ and $x = \kappa(h_{c\lambda_i},h_{c\lambda_i}) =  \kappa(h_{c\lambda_j},h_{c\lambda_j})$, $y =\kappa(h_{c\lambda_i},h_{c\lambda_j})$. The field $K=F(\sqrt{a})$ is the quadratic extension over which $L$ is type 1.

\end{lemma}

\begin{proof}
By (\ref{basis2}), an $F$-basis for the quasi-split toral subalgebra corresponding to the orbit $\{\alpha_i, \alpha_j\}$ is
\[\{ h_{c\lambda_i} + h_{c\lambda_j}, \sqrt{a}h_{c\lambda_i}-\sqrt{a}h_{c\lambda_j} \}\]
We can compute the Killing form as follows.
\begin{align*}
(h_{c\lambda_i} + h_{c\lambda_j},h_{c\lambda_i} + h_{c\lambda_j}) &= (h_{c\lambda_i},h_{c\lambda_i})
+2(h_{c\lambda_i}, h_{c\lambda_j}) + (h_{c\lambda_j},h_{c\lambda_j})
\end{align*}
\begin{align*}
(&\sqrt{a}h_{c\lambda_i} -\sqrt{a} h_{c\lambda_j},\sqrt{a}h_{c\lambda_i} -\sqrt{a} h_{c\lambda_j}) = \\ &a(h_{c\lambda_i},h_{c\lambda_i})
-2a(h_{c\lambda_i}, h_{c\lambda_j}) + a(h_{c\lambda_j},h_{c\lambda_j})
\end{align*}
\begin{align*}
(h_{c\lambda_i} + h_{c\lambda_j},\sqrt{a}h_{c\lambda_i} -\sqrt{a} h_{c\lambda_j}) &= \sqrt{a}(h_{c\lambda_i},h_{c\lambda_i})
-\sqrt{a}(h_{c\lambda_j},h_{c\lambda_j})
\end{align*}
Because of the symmetry of the root system $(h_{c\lambda_i},h_{c\lambda_i}) = (h_{c\lambda_j},h_{c\lambda_j})$. Let $x = (h_{c\lambda_j},h_{c\lambda_j})$ and let $y=(h_{c\lambda_i},h_{c\lambda_j})$ giving the Killing form as \[\qform{2(x+y), 2a(x-y)}\]
Since $x,y$ are dimensions of irreducible representations $x,y > 0$. Furthermore roots with positive $\alpha_i$ and $\alpha_j$ coefficients contribute the sum in \ref{kappaij}, while those with only positive $\alpha_i$ (or $\alpha_j$) coefficient contribute only to the sum in \ref{kappaii}. Since $\alpha_i, \alpha_j$ are roots themselves, the latter sum will always be larger giving $x > y$.
\end{proof}
For a real Lie algebra, we let $\kappa$ denote the signature of the Killing form in  $\mathbb{Z}$. We use $\lambda_i$ denote a fundamental dominant weight of the larger Lie algebra $L$ and $\omega_i$ to denote a fundamental dominant weight of the subalgebra $L_J$ with simple roots $\Delta_J$. Let $L$ be a Lie algebra over $\mathbb{R}$ with anisotropic kernel $A$. Then $L$ decomposes as
\[L = Z_H(A) \oplus A \oplus \bigoplus_{\mu \in S^*, \mu \neq 0} (L_{\mu} \oplus L_{-\mu})\]
By Lemma \ref{rootform}, we know that $\kappa|_{(L_{\mu} \oplus L_{-\mu})}$ is hyperbolic so the signature here is 0. The subalgebra $A$ is compact, so $\kappa_{A} = (\dim A)\qform{-1}$ \cite[Proposition 6.6]{Hel}. Since $m^*(L)$, $m^*(A)$ are positive, $\frac{m^*(L)}{m^*(A)} = 1 \in \mathbb{R^*}/\mathbb{R}^{*2}$ and $\kappa|_{A} = (\dim A) \qform{-1}$.  This leaves only $\kappa|_{Z_H(A)}$. Since $A$ is the anisotropic kernel of $L$, $Z_H(A)$ is split or quasi-split.

Suppose $Z_H(A)$ contains non-trivial orbits so that $L$ is type 2. By Lemma \ref{orbit}, the Killing form on a non-trivial orbit is $\qform{2(x+y),-2(x-y)}$ where $x>y>0$. This form is hyperbolic, so its signature is also zero. The trivial orbits of $Z_H(A)$ are split. By Lemma \ref{rootform}, the non-hyperbolic part of the Killing form on a trivial orbit is a positive multiple of the Weyl-invariant bilinear form on coroots, which is positive definite \cite[8.5]{Hum}. This gives the Killing form on a trivial orbit as $\qform{1}$. Since each orbit in $Z_H(A)$ is orthogonal with respect to $\kappa$ we have $\kappa|_{Z_H(A)} = n\qform{1}$ where $n$ is the number of circled single vertices in the Tits index. This gives the Killing form on $L$:
\begin{equation}
\kappa = \text{\footnotesize{(\# of circled single vertices in Tits index)}}-\text{\footnotesize{(dim of the anisotropic kernel)}}
\end{equation}

\section{Inner Type $E_6$}
One of the main goals for this new method of calculating Killing form of isotropic Lie algebras is to be able to give a formula for the Killing form of the lesser understood exceptional Lie algebras. This method allows us to build up a Killing form for an exceptional Lie algebra based on the Killing form of a classical subalgebra. In the $E_6$ case, we utilize our knowledge of its $D_4$ subalgebra. Understanding $D_4$ together with the methods developed in the previous sections allows us to give an explicit formula for any Lie algebra of inner type $E_6$ based on its Rost invariant. The critical case here is the $E_6$ Lie algebra whose anisotropic kernel is of type $D_4$. We first show that we can achieve this case (or the split case) over an odd degree extension for any $^1\!E_6$ Lie algebra.

\begin{lemma}\label{odd}  For any Lie algebra $L$ of type $^1\!E_6$ (inner type) over $F$, there exists an odd degree extension $K/F$ such that the Tits index of $L \otimes_F K$ is one of the following:
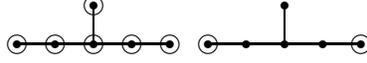
\[
\begin{picture}(2.5,1.1)
\put(0,0){\usebox{\iEpic}}
    \put(0.25,0.25){\circle{\lrad}}
    \put(2.25,0.25){\circle{\lrad}}
    \put(1.25,0.25){\circle{\lrad}}
    \put(1.25,0.75){\circle{\lrad}}
    \put(.75,0.25){\circle{\lrad}}
    \put(1.75,0.25){\circle{\lrad}}
\end{picture}
\begin{picture}(2.5,1.1)
\put(0,0){\usebox{\iEpic}}
    \put(0.25,0.25){\circle{\lrad}}
    \put(2.25,0.25){\circle{\lrad}}
\end{picture} \]\\

\end{lemma}
\begin{proof}
It suffices to prove this in the case $L$ is an algebraic group of type $^1\!E_6$. We first show that there exists an odd-degree extension of $F$ for which the Tits algebras of $L$ are trivial. Recall that $L$ has Tits algebras $A(2) = A(4) = F$ and $A(1) = A(5) = A(3)^{op}= A(6)^{op}$. The Tits algebra $A(1)$ has order dividing 27 in $\Br(F)$  (\cite[6.4.1]{Tits1971})and so there exists an odd degree extension $K$ that splits $A(1)$.

Over $K$ then, $L \cong \Inv(A)$, the group of isometries of the cubic norm form of an Albert algebra $A$ (\cite[Theorem 1.4]{SA}). We will show that there exists an odd degree extension of $K$ for which $A$ is reduced and that over this extension, the Tits index of $L$ has vertex 1 circled. By Tits' classification (\cite{TitsSS}), the Tits index will be one of the two given.

By \cite[40.8]{KMRT}, $A$ is reduced if and only if the invariant $g_3(A) \in H^3(*, \mathbb{Z}/3\mathbb{Z})$ is zero. But $g_3(A)$ is a symbol in $H^3(*, \mathbb{Z}/3\mathbb{Z})$ (\cite[p.303]{Th}), so writing $g_3(A) = (a)\cdot(b)\cdot(c)$, it is enough to show that $\res(a) \in  H^1(*, \mathbb{Z}/3\mathbb{Z}) = \Hom(*,\mathbb{Z}/3\mathbb{Z}) $ is zero over an odd-degree extension of $K$. Let $K_{a}$ be the Galois extension of $K$ corresponding to $\ker(a)$. Then $[K_a:K] = |\im(a)|$ divides 3 and clearly $\res(a) = 0$ in $H^1(K_a, \mathbb{Z}/3\mathbb{Z})$. Therefore, over $K_a$, $A$ is reduced.

We have a bijection between homogeneous varieties $L$-varieties over $F$ and $*$-invariant subsets of the Dynkin diagram \cite[6.4,2]{BT}. Furthermore, such a variety has a point if and only if this subset consists of circled vertices \cite[6.3,1]{BT}. In this case, we have an explicit description of the variety associated to vertex 1 in \cite[7.10]{GC}. It is $\{Kv  \mid  \text{$v\in A$ with $v \neq 0$ and $v^{\sharp} = 0$}\}$. Consider the diagonal matrix $(1,0,0) \in A$. Using the formula for $v^{\sharp}$ given in \cite[p.385, (6)]{Jac3} we have $(1,0,0)^{\sharp} = 0$. Since this element is nonzero, it is in the variety associated to vertex 1 and hence vertex 1 is circled in the Tits index of $L \otimes K_a$. Since $[K_a:K]$ and $[K:F]$ are both odd, we have proven the proposition. \qedhere

%
\end{proof}

\begin{lemma}\label{d4} Let $L$ be a Lie algebra of type $^1\!E_6$ as in Lemma \ref{odd}.
Then the $D_4$ subalgebra is $\frak{so}(q)$, where $q$ is a $3$-fold Pfister form  and $e_3(q) = r(L)$, the mod 2 part of the Rost invariant of $L$ as defined in \cite{GG}.
\end{lemma}

\begin{proof} Again we may prove this in the algebraic group case. Let $L$ be an algebraic group wth  Tits index above and let $D_4$ denote the $D_4$ subgroup with roots $\alpha_2, \alpha_3, \alpha_4, \alpha_5$.
Let $A'(2)$ denote the Tits algebra associated to vertex 2 as a Tits algebra of $D_4$. In the first case, it is clear that $A'(2)$ is split. In the second, $D_4$ is the anisotropic kernel of $L$ and so $A'(2)$ is the same as the Tits algebra $A(2)$ associated to vertex 2 in $L$ \cite[5.5.2,5.5]{Tits1971}. In $L$, $\lambda_2 \in \Lambda_r$ indicating that $A(2)$ is split over $F$ and the irreducible $8$-dimensional representation of $D_4$ with highest weight $\lambda_2$, $V(\lambda_2)$, is defined over $F$. By \cite[5.1.21, 5.1.24 and proof of 2.5.5]{GW}, there exists a non-zero $D_4$-invariant symmetric bilinear form $q$ on $V(\lambda_2)$ and a representation $\pi: D_4 \rightarrow {so}(q)$. Since $D_4$ is simple with $\dim D_4 = \dim {so}(q)$, $\pi$ is an isomorphism. By \cite[6.2]{Tits1971}, $C_0(q) = A(3) \times A(5)$. Furthermore, $\lambda_5 \equiv -\lambda_3 \equiv \lambda_1$ modulo $\Lambda_r$ and since $A(1)$ is split, $C_0(q) = M_8(F) \times M_8(F)$ and $C(q) = M_{16}(F)$ giving the Clifford invariant $c(q) = 1$ (\cite[V.3.12]{Lam}). Since $\dim q = 8$, $\disc(q) = (-1)^4$ and $c(q) = 1$, we have $q \in I^3F$ by Merkuryev's Theorem.

Let $E_6$ denote the split simply connected Lie algebra of type $^1\!E_6$. Then there is a unique $\eta \in H^1(F, E_6)$ such that $E_6$ twisted by $\eta$ is $L$ (\cite[Proposition 31.5]{KMRT}). But $L$ is also isomorphic to the algebra obtained by twisting $E_6$ by $\eta' \in H^1(F, {so}_8) \subset H^1(K, E_6)$ where ${so}_8$ twisted by $\eta'$ is ${so}(q)$ (\cite[Theorem 2.2]{TitsSS}). Furthermore $\eta' \in H^1(F, {\spin}_8)$ since $q$ is Pfister (\cite[31.41]{KMRT}). Then the quadratic form $q$ is uniquely associated to $\eta$ by Arason-Pfister Haupstatz \cite[X.5.1]{Lam}. Consider the following sequence:
\[ \begin{CD}
H^1(F, \spin_8) @>i>> H^1(F,E_6) @>\text{Rost}>> H^{3}(F,\mu_6^{\otimes 2}) @>p>> H^3(F, \mathbb{Z}/2\mathbb{Z})
\end{CD} \]
The Rost multiplier of the inclusion ${\spin}_8 \rightarrow E_6$ is 1 (\cite[2.2]{Gab}). Therefore
\[\text{Rost}_{E_6}(\eta) = \text{Rost}_{{\spin}_8}(\eta') = e_3(q) \in H^3(F,\mathbb{Z}/2\mathbb{Z}).\]
By \cite[Prop. 5.2 and Def. 5.3]{GG}, $\text{Rost}_{E_6}(\eta)$ projected to $H^3(F, \mathbb{Z}/2\mathbb{Z})$ is precisely the $r$-invariant $r(L)$.

 \end{proof}

We now prove Theorem \ref{esix1}.

\begin{proof} Let $L$ be a Lie algebra of type $^1\!E_6$ over $F$ and let $K/F$ be the odd degree extension described in
Lemma \ref{odd}.  We will show that over $K$, $L$ has a subalgebra of type $D_4 = \frak{so}(q)$ where $q$ is a Pfister form defined over $F$.

By Lemma \ref{d4}, $L \otimes K$ has a $D_4$ subalgebra of the form $\frak{so}(q)$ where $q$ is a $K$-Pfister form and $e_3(q) \in H^3(K, \mathbb{Z}/2\mathbb{Z})$ is $r(L \otimes K)= \res(r(L \otimes F))$. Since $\res(r(L \otimes F))$ is a symbol in $H^3(K, \mathbb{Z}/2 \mathbb{Z})$, $r(L \otimes F)$ is a symbol in $H^3(F, \mathbb{Z}/ 2\mathbb{Z})$ (\cite[Proposition 2]{DP}). Therefore $q$ is in fact a Pfister form over $F$.

We now compute the Killing form of $L \otimes K$ using Theorem \ref{main}. By Springer's Theorem, this is isomorphic to the Killing form of $L \otimes F$ since $[K:F]$ is odd. It is a straightforward exercise to show that the Killing form of  $\frak{so}(q)$ for $q=\qform{a_1,a_2,..,a_n}$ is $\kappa \cong \qform{-2(2n-2)}\lambda^2q$, where $\lambda_2q = \perp_{i<j} \qform{a_ia_j}$. In the case $q$ is $3$-fold Pfister form, the Killing form of $\frak{so}(q)$ is
\[\qform{-3}4q_0\]
where $q \cong 1 \perp q_0$.

Let $A = \frak{so}(q)$, the subalgebra of type $D_4$ in this $E_6$. Then $Z_H(A)$ is generated by $h_{3\lambda_1}, h_{3\lambda_6}$ and we compute the Killing form using the methods described in the previous chapter. Again, we use $\lambda_i$ to denote a fundamental dominant weight of $L \otimes K$ and $\omega_i$ to denote a fundamental dominant weight of $L_J \otimes K$. In these calculations we will use $J = \{1,2,3,4,5\}$, $\{2,3,4,5,6\}$, and $\{2,3,4,5\}$ so $L_J \otimes K$ is of type $D_5$, $D_5$, and $D_4$, respectively.

\begin{align*}
\kappa(h_{3\lambda_1}, h_{3 \lambda_1}) &= 18  \dim
V_{D_5}(-\alpha_1) = 18 \dim V_{D_5}(-2\lambda_1+\lambda_3)\\
&= 18 \dim V_{D_5}(-\omega_4)= 18*2^4 = 288
\end{align*}
\begin{align*}
\kappa(h_{3\lambda_6}, h_{3 \lambda_6})
&= 18 \dim V_{D_5}(-\alpha_6)
= 18 \dim V_{D_5}(\lambda_5-2\lambda_6) \\
&= 18 \dim V_{D_5} (\omega_5)
= 18*2^4 = 288\\
\end{align*}
\begin{align*}
\kappa(h_{3\lambda_1}, h_{3 \lambda_6})&= 18 \dim
V_{D_4}(-(\alpha_1+\alpha_3+\alpha_4+\alpha_5+\alpha_6))\\
&=18 \dim V_{D_4}(-\lambda_1+\lambda_2-\lambda_6) = 18 \dim V_{D_4}(\omega_2)\\
&= 18*8=144
\end{align*}

This form diagonalizes to $\qform{2,6}$.
Using Theorem \ref{main} and the fact that $\frac{m(E_6)}{m(D_4)}= \frac{12}{6} = 2$ gives the result. \qedhere

\end{proof}
The table below gives the results for each type of $^1\!E_6$ where $r(L) = 1 \perp q_0$. We need only know the Rost invariant to compute these forms. The reader should compare these streamlined these formulas to Jacobson's \cite{Jac2}. In the split case $r(L) = 4\mathcal{H}$, so we can give the formula with no $q_0$. In the third case, Lemma \ref{odd} shows that the Killing form is isomorphic to the Killing form of the split Lie algebra.

\begin{table}[h]
\begin{center}
\begin{tabular} {|c|c|}
\hline
Tits Index & Killing form \\
\hline
 \begin{picture}(2.5, 1.1)
\put(0,0){\usebox{\iEpic}}
    \put(1.25,0.25){\circle{\lrad}}
    \put(1.25,0.75){\circle{\lrad}}

\end{picture} or more circles & $\qform{1,1,1,1,2,6} \perp 36\mathcal{H}$ \\
\begin{picture}(2.5, 1.1)
\put(0,0){\usebox{\iEpic}}
\end{picture}  or more circles &  $\qform{-1}4q_0 \perp \qform{2,6} \perp 24 \mathcal{H}$\\
\hline
\end{tabular}
\caption{Results for $^1\!E_6$}
\end{center}
\end{table}
\section{Results for Lie algebras of type $^2\!E_6$}\label{esix2}
Theorem \ref{main} also yields results for some outer type $E_6$ Lie algebras. For this section $L$ refers to a Lie algebra of outer type $E_6$ with one of the following Tits indices:
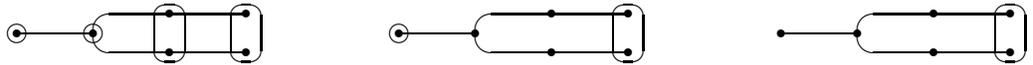
\begin{equation}
\begin{picture}(5, 1) \put(0,0){\usebox{\dEpic}}
\put(1,0.5){\circle{\lrad}} \put(2,0.5){\circle{\lrad}}
\put(3,0.5){\oval(0.4,0.75)} \put(4,0.5){\oval(0.4,0.75)}
\end{picture}
\begin{picture}(5, 1) \put(0,0){\usebox{\dEpic}}
\put(1,0.5){\circle{\lrad}}
 \put(4,0.5){\oval(0.4,0.75)}
\end{picture}
\begin{picture}(5, 1) \put(0,0){\usebox{\dEpic}}
 \put(4,0.5){\oval(0.4,0.75)}
\end{picture}
\end{equation}
As in the inner type case, the subalgebra of type $^2\!D_4$ plays a crucial role. We begin by describing this subalgebra.

\begin{lemma} \label{2d4}
Let $\Gamma = Gal(F_{sep}/F)$ and let $*: \Gamma \rightarrow S_3$ be the map from $\Gamma$ into the automorphism group of the Tits index of $^2\!D_4$ given by the  $*$-action of $\Gamma$. Let $K = (F_{sep})^{\ker *}$. Then $K = F(\sqrt{a})$ for some $a \in \dot{F}$ and
\[^2\!D_4 = \frak{so}(1 \perp \qform{a}q_0)\]
where $1 \perp q_0$ is a Pfister form over $K$.
\end{lemma}

\begin{proof}
Over $F$, the fundamental weight $\lambda_2$ is in the root lattice of $E_6$
so $A(2)$ is split as a Tits algebra of $E_6$. When vertex $2$ is circled in the Tits index of a Lie algebra of type $^2\!E_6$,
 $A(2)$ is split as a Tits algebra of $D_4.$ When $D_4$ is the anisotropic kernel of $E_6$, the Tits algebras for the
vertices in $D_4$ are the same as those for the vertices 2,3,4,5 in
$E_6$ above \cite[p.211]{Tits1971}.
Since $A_2$ is split in each of these cases, $V(\lambda_2)$, the
8-dimensional irreducible representation of $D_4$ with highest
weight $\lambda_2$, is defined over the base
field $F$. There exists a $D_4$ invariant symmetric bilinear form on
$V(\lambda_2)$ by \cite[5.1.21, 5.1.24 and Proof of 2.5.5]{GW} so we
have a representation $\pi: D_4 \rightarrow \frak{so}(q)$ for some
8-dimensional quadratic form $q$. Since $D_4$ is simple, $\pi$ is
injective, and since dim$D_4$=dim$\frak{so}(q)$=28, $\pi$ is an
isomorphism.

The image of the map $*: \Gamma \rightarrow S_3$ has order 2 since our $D_4$ is of type 2 \cite{TitsSS}. Therefore $[K:F] = 2$ and so $K = F(\sqrt{a})$ for some $a \in \dot{F}$. Notice that $L \otimes K$ is of inner type.
We now describe the Tits algebras of $L$ over $K$. Let $B$ be the Tits algebra associated to the orbit $\{\alpha_3, \alpha_5\}$ over $F$ (using the number of the vertices in $E_6$). Let $\iota$ be the non-identity element of $*(\Gamma)$. Then over $K$
\[B \otimes_F K = (B \otimes_K K) \otimes_F K
=B  \otimes_K (K \otimes_F K) = B  \otimes_K (K \times {^{\iota}K})
\]\[B  \otimes_K K \times B  \otimes_K {^{\iota}K}  = B  \times
{^{\iota}B}  = A(3) \times A(5)\]
where $A(3)$ and $A(5)$ are Tits algebras associated to vertices $3,5$ in the inner type $D_4 \subset E_6$.
The even Clifford algebra of $q \otimes K = A(3) \otimes A(5)$ \cite[6.2]{Tits1971}. But we are in the inner case over $K$, so there exists an odd degree extension of $K$ for which all Tits algebras are trivial. (See proof of Lemma \ref{odd}.) Over this extension $K'$ then $C_0(q \otimes K') = M_8(K') \otimes M_8(K')$ and $C(q \otimes K') = M_16(K')$ \cite[V.2.5]{Lam} so $c(q \otimes K') = 1$ and since also $\dim q = 8$, $\disc{q} = 1$, $q \otimes K'$ is a Pfister form. But then $q \otimes K$ is Pfister \cite[Proposition 2]{DP} and
\[q \otimes F = \qform{1} \perp \qform{a}q_0\]
where $1 \perp q_0 = q \otimes K$ \cite[VII.3.3]{Lam}.
\end{proof}

\begin{thm}\label{2esix}
Let $L$ be a Lie algebra of outer type $E_6$ with one of the Tits indices mentioned. The Killing form of $L$ is
\[\kappa \cong \qform{-6}(3q_0 \perp \qform{a}q_0) \perp \qform{6,2a} \perp 24 \mathcal{H}\]
where $\qform{1 \perp q_0}$ is $r(L \otimes K)$.
\end{thm}

\begin{proof} Let $A$ be the subalgebra of type $^2\!D_4$. The ratio of the Coxeter numbers of $L$ and $A$ is 2 so Theorem \ref{main} gives
\[\kappa = 2\kappa_{A} \oplus \kappa|_{Z_H(A)} \oplus 24 \mathcal{H} \]
Since $A = \frak{so}(1 \perp \qform{a}q_0)$, the Killing form of $A$ is
\[\kappa_A = -2(6)\lambda^2q = 3q_0 \perp \qform{a}q_0\]
The quasi-split toral subalgebra $Z_H(A)$ has $F$-basis $\{h_{3\lambda_1}+h_{3\lambda_6},
\sqrt{a}h_{3\lambda_1}-\sqrt{a}h_{3\lambda_6}\}$ \cite[6.9]{BT}. We calculate the Killing form using results from the $^1E_6$ case which are restated below.
\begin{align*}
\kappa(h_{3\lambda_1}, h_{3 \lambda_1}) &= 288 \\
\kappa(h_{3\lambda_6}, h_{3 \lambda_6}) &= 288\\
\kappa(h_{3\lambda_1}, h_{3 \lambda_6})& =144
\end{align*}
The Killing form for $Z_H(A)$ in the type 2 case is then:
\begin{align*}
&\kappa(h_{3\lambda_1}+h_{3\lambda_6},h_{3\lambda_1}+h_{3\lambda_6})= 3(288)
\end{align*}

\begin{align*}
&\kappa(h_{3\lambda_1}+h_{3\lambda_6},\sqrt{a}h_{3\lambda_1}-\sqrt{a}h_{3\lambda_6}) = 0
\end{align*}

\begin{align*}& \kappa(\sqrt{a}h_{3\lambda_1}-\sqrt{a}h_{3\lambda_6},\sqrt{a}h_{3\lambda_1}-\sqrt{a}h_{3\lambda_6})=a(288)
\end{align*}
This form diagonalizes to
\[\qform{6,2a}.\]
Combining these results gives the formula for outer type $E_6$.
\end{proof}

In the quasi-split case, $1 \perp q_0 = 4\mathcal{H}$ so $3q_0 = \qform{-1} \perp 3 \mathcal{H}$ giving the Killing form on $^2D_4$ as \[\qform{-1}\qform{1,1,1,a} \perp 12 \mathcal{H}\] The table below summarizes the Killing form for the other cases.

\begin{table}[h] \label{t2}
\begin{center}
\begin{tabular}{|c|c|}
\hline
Tits Index & Killing form \\
\hline
\begin{picture}(3,1) \put(0,0){\usebox{\dEpic}}
\put(1,0.5){\circle{\lrad}} \put(2,0.5){\circle{\lrad}}
\put(3,0.5){\oval(0.4,0.75)} \put(4,0.5){\oval(0.4,0.75)}
\end{picture}\hspace{110pt} & $\qform{6}\qform{1,1,1,a} \perp \qform{6,2a} \perp 36\mathcal{H}$\\
\begin{picture}(3, 1) \put(0,0){\usebox{\dEpic}}
 \put(4,0.5){\oval(0.4,0.75)}
\end{picture} \hspace{40pt} or more circles &  $ \qform{-6}(3q_0 \perp \qform{a}q_0)  \perp \qform{6,2a} \perp 24\mathcal{H}$\\
\hline
\end{tabular}
\caption{Results for $^2E_6$}
\end{center}
\end{table}

Garibaldi and Petersson \cite[11.1]{GP} describe a bijection between pairs $(C,K)$ of octonion $F$-algebras $C$ and quadratic extensions $K$ and simply connected groups of type $E_6$ as in Table \ref{t2}. In the above formulas $1 \perp q_0$ is the norm form of $C$ and $K = F(\sqrt{a})$.

It should also be noted that Jacobson's methods also give the above results, although his formulas are much more complicated (see \cite[p.114]{Jac2}) and one needs the main results of \cite{GP} to connect his formulas to the Tits index.

\section{Results for Lie algebras of type $E_7$}
For Lie algebras of type $E_7$ we have only inner type. Theorems \ref{main} and \ref{esix1} give us immediate results for isotropic Lie algebras of type $E_7$ when the anisotropic kernel is contained in the subalgebra of type $E_6$. The Killing form for these Lie algebras is obtained by letting $A = E_6$ in Theorem \ref{main} and using the results from Theorem \ref{esix1} to give the Killing form on $A$.

\begin{thm}\label{eseven}
Let $L$ be an isotropic Lie algebra of type $E_7$ whose Tits index is
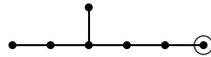
\[\begin{picture}(3, 1.1)
\put(0,0){\usebox{\viipic}}
 \put(2.75,0.25){\circle{\lrad}}
\end{picture}\]
or has more circles. Let $q_0$ denote the invariant $r(E_6)$.  Then the Killing form on $L$ is isomorphic to
\[\qform{2,1,1} \perp 4\qform{-1}q_0 \perp 51 \mathcal{H} \]
\end{thm}

\begin{proof}Let $A = E_6$.
$Z_H(A)$ is a one dimensional $F$-split toral subalgebra corresponding to vertex $7$. A basis for $Z_H(A)$ is just $h_{2\lambda_7}$ and the Killing form on $Z_H(A)$ can be computed as follows.
\begin{align*}
\kappa(h_{2\lambda_7},h_{2\lambda_7}) &=2(2^2) \dim
V_{E_6}(-\alpha_7) = 8 \dim V_{E_6}(\lambda_6-2\lambda_7)\\
& = 8 \dim V_{E_6}(\omega_6) = 8*27 = 216 \\
& \equiv 6 \hspace{1pt} \text{mod $F^{*2}$}
\end{align*}
The Coxeter number of $E_7$ is 18 and the Coxeter number of $E_6$ is 12 so Theorem \ref{main} and the computations above give

\[\kappa = \qform{\frac{3}{2}}(\qform{-24}4q_0 \perp \qform{2,6} \perp 24\mathcal{H}) \perp \qform{6} \perp 27 \mathcal{H} \]

Simplifying the form and reducing modulo squares gives
\[\kappa = 4\qform{-1}q_0 \perp \qform{3,1,6} \perp 51 \mathcal{H}\]
But
\[\qform{3,6} \cong \qform{1,2} \]
since they have the same discriminant and both represent 9 \cite[I.5.1]{Lam}.
This gives
\[\kappa =  \qform{2,1,1} \perp 4\qform{-1q_0} \perp 51 \mathcal{H} \qedhere \]

\end{proof}

Table \ref{t3} summarizes the results for $E_7$. Here $1 \perp q_0$ is the $r$-invariant of the $E_6$ subalgebra. In the split case $1 \perp q_0 = 4 \mathcal{H}$ so $q_0 = \qform{-1} \perp 3 \mathcal{H}$.

\begin{table}[ht]\label{t3}
\begin{center}
\begin{tabular} {|c|c|}
\hline
Tits Index & Killing form \\
\hline
\begin{picture}(2.5, 1.1)
\put(0,0){\usebox{\viipic}}
    \put(0.25,0.25){\circle{\lrad}}
    \put(2.25,0.25){\circle{\lrad}}
    \put(1.25,0.25){\circle{\lrad}}
    \put(1.25,0.75){\circle{\lrad}}
    \put(.75,0.25){\circle{\lrad}}
    \put(1.75,0.25){\circle{\lrad}}
        \put(2.75,0.25){\circle{\lrad}}
\end{picture} \hspace{75pt} & $\qform{2,1,1,1,1,1,1} \perp 63\mathcal{H}$ \\
%
%
%
%

\begin{picture}(2.5, 1.1)
\put(0,0){\usebox{\viipic}}
 \put(2.75,0.25){\circle{\lrad}}
\end{picture} \hspace{10pt} or more circles
&  $\qform{2,1,1} \perp 4\qform{-1}q_0 \perp 51\mathcal{H}$\\
\hline
\end{tabular}
\caption{Results for $E_7$}
\end{center}
\end{table}

\section{Acknowledgements}
The author thanks her thesis advisor, Skip Garibaldi, as well as J-P. Serre, R. Parimala and D. Hoffmann for their comments.

\end{document}